\documentclass[12pt,a4paper]{article}
\usepackage[T2A]{fontenc}
\usepackage[cp1251]{inputenc}
\usepackage{cmap}
\usepackage[english]{babel}
\usepackage{amsmath,amsfonts,amssymb, amsthm}
\usepackage{graphicx}
\usepackage{cite}
\usepackage{xcolor}
\usepackage{hyperref}

\theoremstyle{plain}
\newtheorem{theorem}{Theorem}
\newtheorem{lemma}{Lemma}[section]

\newtheorem{corollary}{Corollary}[theorem]

\theoremstyle{definition}
\newtheorem{note}{Remark}[section]
\newtheorem{example}{Example}[section]
\newtheorem{definition}{Definition}[section]

\usepackage[auth-sc]{authblk}

\title{On $\mathrm{K}\mathfrak F$-subnormality and submodularity \\in a finite group}
\author{Victor S. Monakhov}
\author{Irina L. Sokhor}
\affil{Department of Mathematics and Technologies of Programming,\\
Francisk Skorina Gomel State University, Belarus}
\affil{victor.monakhov@gmail.com, irina.sokhor@gmail.com}
\date{}
\begin{document}
\maketitle

{\footnotesize
\abstract{ Let $\mathfrak F$ be a formation and let $G$ be a group.
A subgroup~$H$ of~$G$ is $\mathrm{K}\mathfrak F$-subnormal (submodular) in $G$
if there is a subgroup  chain $H=H_0\le \  H_1 \le  \  \ldots \le H_i \leq H_{i+1}\le \ldots   \le \  H_n=G$
such that for every $i$ either $H_{i}$ is normal in $H_{i+1}$ or $H_{i+1}^\mathfrak{F} \le H_i$
($H_i$ is a modular subgroup of $H_{i+1}$, respectively). We prove that
a primary subgroup $R$ of a group $G$ is submodular in $G$ if and only if
$R$ is  $\mathrm{K}\mathfrak U_1$-subnormal in $G$.
Here $\mathfrak{U}_1$ is the class of all supersolvable groups of square-free exponent.
In addition, for a solvable subgroup-closed formation $\mathfrak{F}$,
every solvable $\mathrm{K}\mathfrak{F}$-subnormal subgroup of a group~$G$ is contained
in the solvable radical of $G$. }

}


\section{Introduction}
All groups in this paper are finite.
The formations of all nilpotent, supersolvable, solvable groups and all finite groups
are denoted by~$\mathfrak N$, $\mathfrak U$ and $\mathfrak S$,
respectively.

Let $\mathfrak{F}$ be a formation.
We say that $\mathfrak F$ is solvable if $\mathfrak F\subseteq \mathfrak S$.
The intersection of all normal subgroups of a group $G$ with quotient in $\mathfrak{F}$
is called the $\mathfrak{F}$-residual of $G$ and denoted by $G^\mathfrak{F}$.

The group structure in many respects depends on the way of embedding of subgroups.
In particular, subnormal subgroups have a defining value in the structure
of non-simple groups. It is well known that the concept of subnormality
is an extension of the concept of normality and, in contrast with normality,
it is a transitive relation. In the context of the formation theory,
Kegel~\cite{kegel} proposed the following extension of subnormality.

\begin{definition}
Let $\mathfrak F$ be a formation and let $G$ be a group.
A subgroup~$H$ of~$G$ is $\mathrm{K}\mathfrak F$-subnormal in $G$
if there is a subgroup  chain
\begin{equation}\label{eqkfsn}
H=H_0\le \  H_1 \le  \  \ldots \le H_i \leq H_{i+1}\le \ldots   \le \  H_n=G
\end{equation}
such that either $H_{i}$ is normal in $H_{i+1}$ or $H_{i+1}^\mathfrak{F} \le H_i$ for every $i$.
\end{definition}

It is clear that a subnormal subgroup of a group~$G$ is $\mathrm{K}\mathfrak F$-subnormal in $G$
for any formation~$\mathfrak F$, and every $\mathrm{K}\mathfrak N$-subnormal subgroup of $G$
is subnormal in $G$.

The results related to these concept are presented in
monographs~\cite{schem,dh,ks,ballcl}.

The concept of a modular subgroup came from the lattice theory.
A subgroup $H$ of a group $G$ is modular in $G$ if $H$ is a modular element of the
lattice of subgroups of $G$. R.~Schmidt~\cite{sch94} conducted a detailed analysis of
the behavior of modular subgroups in a group. Like normality, modularity is not a transitive relation.
The concept of submodularity is an extension of the concept of modularity, and
it is a transitive relation.

\begin{definition}[{\cite[p. 546]{sch94}}]
A subgroup $H$ of a group $G$ is submodular in $G$ if there is a subgroup chain~\eqref{eqkfsn}
such that $H_i$ is a modular subgroup of $H_{i+1}$ for every~$i$.
\end{definition}

In the present paper, we investigate groups with some $\mathrm{K}\mathfrak{F}$-subnormal subgroups,
and we also establish a connection between submodularity  and $\mathrm{K}\mathfrak{F}$-subnormality.
In particular, we obtained the following results.

\begin{theorem}\label{tsm}
A primary subgroup $Q$ of a group $G$ is submodular in $G$ if and only if
$Q$ is  $\mathrm{K}\mathfrak U_1$-subnormal in~$G$.
\end{theorem}

\begin{theorem}\label{t2}
Let $\mathfrak{F}$ be a solvable subgroup-closed formation.
If $A$ is a solvable $\mathrm{K}\mathfrak{F}$-subnormal subgroup of a group~$G$,
then $A\le G_{\mathfrak S}$.
\end{theorem}

Here $G_{\mathfrak{S}}$ is the solvable radical of a group $G$, i.\,e. the largest normal
solvable subgroup of~$G$. A primary group is a group of prime power order.
By $\mathfrak{U}_1$ we denote the class of all supersolvable groups of square-free exponent.

\begin{corollary}\label{ct2}
Let $\mathfrak{F}$ be a solvable subgroup-closed formation.
If $A$ is a solvable $\mathrm{K}\mathfrak{F}$-subnormal subgroup of a group~$G$,
then $\langle A,B\rangle$ is solvable for any solvable subgroup~$B$ of $G$.
\end{corollary}

\begin{corollary}\label{corgs}
If $A$ is a solvable submodular subgroup of a group $G$, then ${A\le G_{\mathfrak{S}}}$.
In particular, if $A$ is a solvable submodular subgroup a group~$G$,
then $\langle A,B\rangle$ is solvable for any solvable subgroup~$B$ of $G$.
\end{corollary}

\section{Preliminaries}\label{spr}

If $X$ is a subgroup (proper subgroup, normal subgroup, maximal subgroup)  of a group~$Y$,
then we write $X\le Y$ ($X<Y$, $X\lhd Y$, $X\lessdot  \ Y$, respectively).

If $X\le Y$, then $X_Y=\cap _{y\in Y}X^y$ is the core of~$X$ in~$Y$.
For a formation $\mathfrak F$ and $X\le Y$,
the inclusion $Y^\mathfrak F\le X$ is equivalent to~$Y/X_Y\in \mathfrak F$.
Obviously, $G^\mathfrak H\le G^\mathfrak F$ for formations $\mathfrak F\subseteq \mathfrak H$,
in particular, $G^\mathfrak S\le G^\mathfrak U\le G^\mathfrak N$.

Remind the properties of $\mathrm{K}\mathfrak{F}$-subnormal
subgroups that we use.

\begin{lemma}[{\cite[6.1.6, 6.1.7, 6.1.9]{ballcl}}]\label{lfsn}
Let $\mathfrak{F}$ be a formation, let $H$ and $L$ be subgroups of a group $G$,
and let $N\lhd G$. The following statements hold.

$(1)$~If $H$ is $\mathrm{K}\mathfrak{F}$-subnormal in  $L$
and $L$ is $\mathrm{K}\mathfrak{F}$-subnormal in  $G$,
then $H$ is $\mathrm{K}\mathfrak{F}$-subnormal in  $G$.

$(2)$~If $N \leq H$ and $H/N$ is $\mathrm{K}\mathfrak{F}$-subnormal in
$G/N$, then $H$ is $\mathrm{K}\mathfrak{F}$-subnormal in  $G$.

$(3)$~If $H$ is $\mathrm{K}\mathfrak{F}$-subnormal in  $G$,
then $HN/N$ is $\mathrm{K}\mathfrak{F}$-subnormal in  $G/N$
and $HN$ is $\mathrm{K}\mathfrak{F}$-subnormal in  $G$.

$(4)$~If $\mathfrak{F}$ is a subgroup-closed formation and $H$ is
$\mathrm{K}\mathfrak{F}$-subnormal in  $G$, then $H\cap L$ is
$H \cap L$ is $\mathrm{K}\mathfrak{F}$-subnormal in  $L$.
\end{lemma}

\begin{lemma}\label{lkfmax}
Let $\mathfrak F$ be a subgroup-closed formation.
If a proper subgroup $H$ of a group $G$ is $\mathrm{K}\mathfrak F$-subnormal in~$G$,
then there is a subgroup~$M$ of~$G$  such that~$H\le M$, $H$ is $\mathrm{K}\mathfrak F$-subnormal in~$M$,
and either $M$ is normal in~$G$ and~$G/M$ is a simple group or $M$ is a maximal subgroup of~$G$
and~$G/M_G\in \mathfrak F$.
\end{lemma}

\begin{proof}
By definition, there is a subgroup chain~\eqref{eqkfsn} such that either $H_{i}$ is normal in $H_{i+1}$
or $H_{i+1}^\mathfrak{F} \le H_i$ for every $i$. Let $H_{n-1}$ be normal in $G$.
In that case, there is a maximal normal subgroup $M$ of $G$ such that $H_{n-1}\leq M$.
In view of Lemma~\ref{lfsn}\,(4), $H$ is $\mathrm{K}\mathfrak F$-subnormal in $M$ and $G/M$ is simple.
Assume that $G^\mathfrak{F} \leq H_{n-1}\leq M\lessdot G$. In that case,
$G^\mathfrak{F}\leq M_G$ and $G/M_G\in\mathfrak{F}$.
In view of Lemma~\ref{lfsn}\,(4), $H$ is $\mathrm{K}\mathfrak F$-subnormal in $M$.
\end{proof}

\begin{note}
In the sequel, for a $\mathrm{K}\mathfrak F$-subnormal subgroup~$H$ of a group~$G$,
we use~$H^\star$ to denote a maximal $\mathrm{K}\mathfrak{F}$-subnormal subgroup of~$G$
that contains~$H$. By Lemma~\ref{lkfmax}, $H$ is $\mathrm{K}\mathfrak F$-subnormal in~$H^\star$
and either~$H^\star$ is normal in $G$ and $G/H^\star$ is a simple group
or $H^\star$ is a maximal subgroup of $G$ and $G/H^\star_G\in \mathfrak F$.
We write $H^\star_G$ instead of  $(H^\star)_G$ for short.
\end{note}

\section{Submodularity and $\mathrm{K}\mathfrak {U}_1$-subnormality}\label{ssmod}

The following lemma contains the well-known properties of submodular subgroups.

\begin{lemma}[{\cite[Lemma~1]{zim89}}]\label{lsubm}
Let $G$ be a group, let $H$ and $K$ be subgroups of~$G$ and let $N$ be a normal subgroup of~$G$.

$(1)$~If $H$ is submodular in $K$ and $K$ is submodular in $G$, then $H$ is submodular in~$G$.

$(2)$~If $H$ is submodular in $G$, then $H\cap K$ is submodular in $K$.

$(3)$~If $H/N$ is submodular in $G/N$, then $H$ is submodular in $G$.

$(4)$~If $H$ is submodular in $G$, then $HN/N$ is submodular in $G/N$ and $HN$ is submodular in $G$.

$(5)$~If $H$ is subnormal in $G$, then $H$ is submodular in $G$.
\end{lemma}

Notice that a subgroup $M$ of a group $G$ is a maximal modular subgroup
if $M$ satisfies the following conditions:
$M$ is a proper subgroup of~$G$;
$M$ is a modular subgroup of~$G$;
if $M<K<G$, then $K$ is not modular in~$G$.

\begin{lemma}[{\cite[Lemma~5.1.2]{sch94}}]\label{l512}
If $M$ is a maximal modular subgroup of a group $G$, then either $M$ is normal in $G$ and
$G/M$ is simple or $G/M_G$ is non-abelian of order $pq$ for primes $p$ and~$q$.
\end{lemma}

\begin{lemma}\label{tsm1}
Every submodular subgroup is $\mathrm{K}\mathfrak{U}_1$-subnormal in a group.
\end{lemma}

\begin{proof}
Use induction on the order of a group. Let $H$ be a submodular subgroup of a group $G$.
In that case, there is a maximal modular subgroup $M$ of $G$ such that $H$ is submodular in  $M$.
By induction, $H$ is $\mathrm{K}\mathfrak{U}_1$-subnormal in $M$. If $M$ is normal in $G$,
then $M$ is $\mathrm{K}\mathfrak{U}_1$-subnormal in $G$. Assume that $M$ is not normal in  $G$.
Then $G/M_G\in\mathfrak{U}_1$ in view of Lemma~\ref{l512} and $M$ is
$\mathrm{K}\mathfrak{U}_1$\nobreakdash-sub\-nor\-mal in~$G$.
Therefore $H$ is $\mathrm{K}\mathfrak{U}_1$-subnormal in~$G$ by Lemma~\ref{lfsn}\,(1).
\end{proof}

\begin{example}
The Frobenius group $F_7=C_7\rtimes C_6\in\mathfrak{U}_1$, hence every subgroup of $F_7$ is
$\mathfrak{U}_1$-subnormal in $F_7$, but at the same time the maximal subgroup $C_6$
is not submodular in $F_7$. Therefore the converse of Lemma~\ref{tsm1} is  not true in general.
\end{example}


\begin{proof}[Proof of Theorem~{\upshape\ref{tsm}}]
Use induction on the order of a group. If $Q$ is a submodular primary subgroup of a group~$G$,
then $Q$ is $\mathrm{K}\mathfrak{U}_1$-subnormal in~$G$ by Lemma~\ref{tsm1}.

Conversely, assume that $Q$ is a $\mathrm{K}\mathfrak{U}_1$-subnormal $q$-subgroup of~$G$
for a prime $q\in\pi(G)$. By induction, $Q$ is submodular in $Q^\star $. If $Q^\star $ is
normal in $G$, then $Q^\star $ is modular in $G$, and $Q$ is submodular in $G$
by Lemma~\ref{lsubm}\,(1).

Let $G/Q^\star_G\in \mathfrak{U}_1$. Suppose that $Q^\star_G\neq 1$.
By Lemma~\ref{lfsn}\,(3), $QQ^\star_G/Q^\star_G$ is $\mathrm{K}\mathfrak{U}_1$-subnormal
in $G/Q^\star_G$, and by induction, $QQ^\star_G/Q^\star_G$ is submodular in $G/Q^\star_G$.
Hence $QQ^\star_G$ is  submodular in $G$ by Lemma~\ref{lsubm}\,(3). By Lemma~\ref{lfsn}\,(4),
$Q$ is $\mathrm{K}\mathfrak{U}_1$-subnormal in $QQ^\star_G$. Since $QQ^\star_G\leq Q^\star <G$,
we get by induction that $Q$ is submodular in $QQ^\star_G$.
Consequently, $Q$ is submodular in $G$ by Lemma~\ref{lsubm}\,(1).

Now assume that $Q^\star_G=1$. In that case, $G\in \mathfrak{U}_1$, in particular,
$G$ is a supersolvable primitive group. Therefore in view of~\cite[Theorem~1.1.7,
1.1.10]{ballcl},
$G=F(G)\rtimes Q^\star$ and $F(G)$ is a Sylow $r$-subgroup for $r=\max\pi(G)$, $|F(G)|=r>q$.
By Lemma~\ref{lfsn}\,(3), $QF(G)/F(G)$ is $\mathrm{K}\mathfrak{U}_1$-subnormal in $G/F(G)$.
Consequently,  $QF(G)/F(G)$ is submodular in $G/F(G)$ by induction and $QF(G)$ is submodular in $G$.
By Lemma~\ref{lfsn}\,(4), $Q$ is $\mathrm{K}\mathfrak{U}_1$-subnormal in $QF(G)$.
If $QF(G)<G$, then $Q$ is submodular in $QF(G)$ by induction, and so $Q$ is submodular in $G$.
Let $G=F(G)\rtimes Q$. In that case, $Q$ is a cyclic subgroup and $|G|=rq$,
since $G\in \mathfrak{U}_1$. Consequently, $Q$ is submodular in~$G$.
\end{proof}


\begin{proof}[Proof of Theorem~{\upshape\ref{t2}}]
Obviously, we can assume that~$A\ne 1$ and~$G\notin \mathfrak S$.
Use induction on the order of a group.  If $N$ is a non-trivial normal subgroup of~$G$,
then~$AN/N$ is a solvable $\mathrm{K}\mathfrak{F}$-subnormal subgroup of~$G/N$
and~$AN/N\le (G/N)_{\mathfrak S}$ by induction. Denote $H/N=(G/N)_{\mathfrak S}$.
If $N\in \mathfrak S$, then $H\in \mathfrak S$ and~$A\le AN\le H\le G_{\mathfrak S}$,
i.\,e. the statement is true. Therefore we can assume that~$G_{\mathfrak S}=1$,
and so $G$ has no non-trivial subnormal solvable subgroups.

By induction, we have $A\le A^\star_{\mathfrak S}$.
If~$A^\star$ is normal in~$G$, then $A^\star_{\mathfrak S}$ is a non-trivial subnormal
solvable subgroup, a contradiction. Therefore $A^\star$ is a non-normal maximal subgroup of $G$
and~$G/A^\star_G\in \mathfrak S$ by Lemma~\ref{lkfmax}. Since $A^\star_{\mathfrak S}$ and
$A^\star_G$ are non-identity normal subgroups of $A^\star$ and $A^\star_{\mathfrak S}$ is normal in~$G$,
we deduce that $A^\star_{\mathfrak S}\cap A^\star_G$ is a subnormal solvable subgroup of~$G$,
$A^\star_{\mathfrak S}\cap A^\star_G=1$ and~$A\le A^\star_{\mathfrak S}\le C_G(A^\star_G)$.

Denote $C=C_G(A^\star_G)$. If~$C<G$, then $A\le C_{\mathfrak S}$ by induction, and
$C_{\mathfrak S}$~ is a non-trivial subnormal solvable subgroup of~$G$, a contradiction.
Hence $C=G$ and~$A^\star_G=Z(G)\in\mathfrak{S}$. Consequently, $G\in \mathfrak S$,
a contradiction.
\end{proof}


\begin{proof}[Proof of Corollary~{\upshape\ref{ct2}}]
Since $A\leq G_\mathfrak{S}$ by Theorem~\ref{t2}, we have $\langle A,B\rangle\leq G_\mathfrak{S} B$.
As $G_\mathfrak{S}\in \mathfrak{S}$  and
$G_\mathfrak{S} B/G_\mathfrak{S}\cong B/B\cap G_\mathfrak{S}\in \mathfrak{S}$,
we deduce that $\langle A,B\rangle$ is solvable.
\end{proof}


\begin{proof}[Proof of Corollary~{\upshape\ref{corgs}}]
By Lemma~\ref{tsm1}, $A$ is $\mathrm{K}\mathfrak{U}_1$-subnormal in $G$.
By Theorem~\ref{t2}, ${A\le G_{\mathfrak{S}}}$. Since $\langle A,B\rangle \le G_{\mathfrak{S}}B$
and $G_{\mathfrak{S}}B$ is solvable, we get $\langle A,B\rangle$ is solvable.
\end{proof}

\end{document}